\documentclass[a4paper,10pt,leqno]{amsart}

\usepackage{graphicx}
\usepackage{amsmath}
\usepackage{amsfonts}
\usepackage{amssymb}
\usepackage{mathtools}
\usepackage[all,cmtip]{xy}
\usepackage{amsthm}
\usepackage{comment}
\usepackage{enumitem}

\usepackage{epsfig}
\usepackage[utf8]{inputenc}
\usepackage[colorinlistoftodos]{todonotes}
\usepackage[capitalise]{cleveref}

\makeatletter
\providecommand\@dotsep{5}
\def\listtodoname{List of Todos}
\def\listoftodos{\@starttoc{tdo}\listtodoname}
\makeatother

\newtheorem{theorem}{Theorem}[section]

\newtheorem{lemma}[theorem]{Lemma}

  \theoremstyle{definition}

\newtheorem{remark}[theorem]{Remark}

\newcommand{\dbC}{\mathbb{C}}

\newcommand{\dbH}{{\mathbb H}}

\newcommand{\dbQ}{\mathbb{Q}}
\newcommand{\dbR}{\mathbb{R}}
\newcommand{\dbS}{\mathbb{S}}

\newcommand{\dbZ}{\mathbb{Z}}

\newcommand{\calF}{{\mathcal F}}

\newcommand{\calM}{{\mathcal M}}
\newcommand{\calO}{{\mathcal O}}

\newcommand{\nbeq}{\begin{equation}}
\newcommand{\neeq}{\end{equation}}
\newcommand{\beq}{\begin{equation*}}
\newcommand{\eeq}{\end{equation*}}

\newcommand{\barf}{\underline{\mathrm{F}}}

\DeclareMathOperator{\cd}{cd}
\DeclareMathOperator{\vcd}{vcd}
\DeclareMathOperator{\cdfin}{\underline{cd}}

\DeclareMathOperator{\sing}{sing}

\begin{document}

\title[Groups satisfying certain properties on their finite subgroups]{On the dimension of groups that satisfy 
certain conditions on their finite subgroups}

\author{Luis Jorge S\'anchez Salda\~na}
\address{Departamento de Matemáticas, Facultad de Ciencias, Universidad Nacional Autónoma de México, Circuito Exterior S/N, Cd. Universitaria, Colonia Copilco el Bajo, Delegación Coyoacán, 04510, México D.F., Mexico}
\email{luisjorge@ciencias.unam.mx}



\date{}



\begin{abstract}
We say a group $G$ satisfies properties (M) and (NM) if every non-trivial finite subgroup of $G$ is contained in a unique maximal finite subgroup, and every non-trivial finite maximal subgroup is self-normalizing. We prove that the Bredon cohomological dimension and the virtual cohomological dimension coincide for groups that admit a cocompact model for $\underline{E}G$ and satisfy properties (M) and (NM). Among the examples of groups satisfying these hypothesis are cocompact and arithmetic Fuchsian groups, one relator groups, the Hilbert modular group and $3$-manifold groups.
\end{abstract}
\maketitle

\section{Introduction}

For a group $G$ consider the following properties:
\begin{itemize}
\item[(M)] Every non-trivial finite subgroup of $G$ is contained in a unique maximal finite subgroup of $G$.
\item[(NM)] If $M$ is a non-trivial maximal finite subgroup of $G$ then $N_G(M)=M$, where $N_G(M)$ denotes the normalizer of $M$ in $G$.
\end{itemize}

In this paper $\calF$ will always denote the \emph{family of finite subgroups} of $G$. A \emph{model for the classifying space} $E_\calF G$ (usually denoted $\underline{E}$G), is a $G$-CW-complex $X$ such that every isotropy group is finite and the fixed point set $X^H$ is contractible for every finite subgroup $H$ of $G$. Equivalently, $X$ is a model for $\underline{E}G$ if for every $G$-CW-complex $Y$ with finite isotropy groups, there exists a map, unique up to $G$-homotopy, $Y\to X$. In particular any two models for $\underline{E}G$ are $G$-homotopy equivalent. We say \emph{$G$ is of type $\barf$} if $G$ admits a cocompact model for $\underline{E}G$. 

The \emph{orbit category $\calO_\calF G$ of $G$ with respect to $\calF$} is the category of homogeneous $G$-sets $G/F$, where $F\in\calF$, and morphisms are given by $G$-maps. A Bredon module is a contravariant functor from $\calO_\calF G$ to the category of abelian groups, and a morphism of Bredon modules is a natural transformation. The category of Bredon modules, for $G$ and $\calF$ choosen, is abelian with enough projectives. The \emph{Bredon cohomological dimension} (or proper cohomological dimension) $\cdfin(G)$ of $G$ is the length of the shortest projective resolution of the constant module $\dbZ_\calF$. If $G$ is torsion free then $\cdfin(G)$ is the classical cohomological dimension $\cd(G)$ of $G$.

On the other hand, provided that $G$ is a virtually torsion free group, the \emph{virtual cohomological dimension} $\vcd(G)$ of $G$ is, by definition, the cohomological dimension of a finite index torsion free group. By a well-known theorem due to Serre $\vcd(G)$ does not depend on the finite index subgroup of $G$ that we choose, hence $\vcd(G)$ is a well-defined invariant of $G$.

For every group $G$ we have the following inequality $\vcd(G)\leq\cdfin(G)$. This inequality can be strict due to examples constructed in \cite{LN03,LP17,DS17}. On the other hand $\vcd(G)=\cdfin(G)$ for for the following classes of groups: elementary amenable groups of type $\mathrm{FP}_\infty$ \cite{KMPN09}, $\mathrm{SL}_n(\dbZ)$ \cite{Ash84}, $\mathrm{Out}(F_n)$ \cite{Vo02}, the mapping class group of any surface with boundary components and punctures \cite{AMP14}, any lattice in a classical simple Lie group \cite{ADMS17}, any lattice in the group of isometries of a symmetric space of
non-compact type without Euclidean factors \cite{La19}, and groups acting chamber transitively on a Euclidean building \cite{DMP16}. 

In this note we prove $\vcd(G)=\cdfin(G)$ for groups of type $\barf$ that satisfy properties (M) and (NM). This implies, due to the main theorem of \cite{LM00}, the existence of a cocompact model for $\underline{
E}G$ of dimension $\max\{3,\vcd (G)\}$. Among the examples of groups that satisfy these properties we have:  cocompact Fuchsian groups and arithmetic Fuchsian groups, one relator groups, the Hilbert modular group, and $3$-manifold groups.  It is worth noticing that the class of groups for which our main theorem applies is closed under taking free products, this is a particular case of \Cref{thm:graph:of:groups}.

The proof is very short and relies on \cite[Corollary~3.4]{ADMS17} (\cref{second:criterion} below), which reduces the proof of the main theorem (\Cref{thm:M:NM:vcd:gd}) to computing the dimension of certain fixed point sets of a cocompact $X$ model for $\underline{E}G$. For groups satisfying properties (M) and (NM) we prove that $X$ can be chosen in such a way that the relevant fixed point sets are one-point spaces (\Cref{lemma:collapsing:fixpoints3}), hence they have dimension $0$. 

\subsection*{Acknowledgements}
 The author wishes to  thank Jean-Fran\c cois Lafont for  feedback on a draft of this paper and the referee for very useful comments. This work was funded by the Mexican Council of Science and Technology via the program \textit{Estancias postdoctorales en el extranjero}, and  by the NSF, via grant DMS-1812028.

\section{Preliminaries and main theorem}

We have the following  criterion that we will use to prove $\vcd(G)=\cdfin(G)$ under the hypothesis of our main theorem.

If $G$ acts on a space $X$, and $A$ is an element of $G$, we denote $X^A$ the set consisting of all elements of $X$ fixed by $A$. In a similar way we define $X^H$ for a subgroup  $H$  of $G$ we define .

\begin{theorem}\label{second:criterion}\cite[Corollary~3.4]{ADMS17}
Let $G$ be a virtually torsion-free group of type $\barf$ with a cocompact model $X$ for $\underline{E}G$. Let $K$ be the kernel of the $G$-action on $X$. If $\dim (X^A)< \vcd(G)$ for every finite order element $A$ of $G\setminus K$, then $\vcd(G)=\cdfin(G)$.
\end{theorem}

\begin{remark}
Note that if $G$ is as in \cref{second:criterion}, then $\vcd(G)$ is finite. Let $H$ be a finite index torsion-free subgroup of $G$. Then a cocompact model $X$ for $\underline{E}G$ is, by restriction, a cocompact model for $\underline{E}H=EH$. Therefore $\vcd(G)=\cd(H)\leq \dim (X/H)<\infty$.
\end{remark}

Our next lemma characterizes properties (M) and (NM) in terms of the existence of a model for $\underline{E}G$ with small fix point sets. Recall that $\calF$ is the family of finite subgroups of $G$.

\begin{lemma}\label{lemma:M:NM2}
Let $G$ be a group. Then the following to conditions are equivalent
\begin{enumerate}
    \item There exists a model $X$ for $\underline{E}G$ with the property that $X^H$ consists of exactly one point for every non-trivial finite subgroup $H$ of $G$. 
    \item Properties (M) and (NM) are true for $G$.
\end{enumerate}
\end{lemma}
\begin{proof}
Assume $X$ is a model for $\underline{E}G$ such that $X^H=\{x_H\}$ for every non-trivial finite subgroup $H$ of $G$. Let $K$ be a non-trivial finite subgroup of $G$. Then $K$ is contained in the stabilizer $S$ of $x_H$ for a unique $H\in \calF$. Note that $S$ is a maximal finite subgroup of $G$. In fact, if $S$ is contained in $F\in \calF$, then $F$ must fix a unique point $y$ of $X$ since $X^F$ consists of a point. In particular, $S$ fixes $y$. Hence by uniqueness $F$ and $S$ fix the same point of $X$. Therefore $F\leq S$. This proves that the stabilizer of any $x_H$ is a finite maximal subgroup of $G$, for all $H\in \calF$. Hence $G$ satisfies (M). The normalizer $N_G(S)$ acts on $X^H$, hence $N_G(S)\leq S$. Therefore $G$ satisfies (NM).

Assume that $G$ satisfies (M) and (NM). Let $Y$ by any model for $EG$. Let $I$ be the set of  finite maximal subgroups of $G$. Note that $G$ acts on $I$ by conjugation. Moreover, the stabilizer of $M\in I$ is $N_G(M)=M$, since $G$ satisfies (NM).  Then, the join $X=Y*I$ is a model for $\underline{E}G$, where the $G$--action on $X$ is the diagonal action. In fact, $X$ is contractible since it can be seen as the union of copies cones of $Y$ glued all together by their common base. Let $H\in \calF$, then $X^H$ consists of the conic point represented by the unique finite maximal subgroup that contains $H$.
\end{proof}

Our next lemma tells us that we can collapse down to a point the fixed point sets of any cocompact model for $\underline{E}G$.

\begin{lemma}\label{lemma:collapsing:fixpoints3}
Let $G$ be a group of type $\barf$ that satisfies properties (M) and (NM). Then $G$ admits a cocompact model $X$ for $\underline{E}G$ such that $X^H$ consists of exactly one point for every non-trivial finite subgroup $H$ of $G$.
\end{lemma}
\begin{proof}
Let $Y$ be any cocompact model for $\underline{E}G$. Given a point $y\in Y$, we denote by $Gy$ the $G$-orbit of $y$. Denote by $Y_{\sing}$ the subspace of $Y$ consisting of points with non-trivial isotropy. Note that $Y_{\sing}$ is a $G$-CW-subcomplex of $Y$.

We claim that there exist a finite number of points $y_1,\dots,y_m$ of $Y$ such that the disjoint union $Gy_1\sqcup \cdots \sqcup Gy_m$ is a $G$-deformation retract of $Y_{\sing}$.

By \cref{lemma:M:NM2} there is a model $Z$ for $\underline{E}G$ such that $Z^H$ consists of exactly one point for every non-trivial finite subgroup $H$ of $G$. On the other hand we have unique (up to $G$-homotopy) $G$-maps $f\colon Y\to Z$ and $g\colon Z\to Y$ such that $f\circ g$ and $g\circ f$ are $G$-homotopic to the corresponding identity functions. These functions induce $G$-maps $f'\colon Y_{\sing}\to Z_{\sing}$ and $g'\colon  Z_{\sing} \to Y_{\sing}$, and also, by restriction  $f'\circ g'$ and $g'\circ f'$ are $G$-homotopic to the corresponding identity functions. By construction of $Z$, $Z_{\sing}$ is of the form $\bigsqcup_{M\in \calM} Gz_M$, where $\calM$ is the set of representatives of conjugacy classes of maximal finite subgroups of $G$ and the isotropy of $z_M$ is $M$. Since $G$ admits a cocompact model for $\underline{E}G$, by \cite[Theorem~4.2]{Lu00} $G$ has a finite number of conjugacy classes of finite subgroups. Therefore $Z_{\sing}= Gz_1\sqcup \cdots \sqcup Gz_m$ for certain points $z_1,\dots,z_m$ of $Z$. Define $y_i=f'(z_i)$ for $i=1,\dots ,m$. We can conclude that $Gy_1\sqcup \cdots \sqcup Gy_m$ is a $G$-deformation retract of $Y_{\sing}$. Moreover, if $r\colon Y_{\sing} \to Gy_1\sqcup \cdots \sqcup Gy_m$ is the mentioned retraction, then $r^{-1}(gy_i)$ is contractible and consists of all points $x$ of $Y_{\sing}$ such that $G_x\leq G_{gy_i}= gG_{y_i}g^{-1}$. Hence the setwise stabilizer of $r^{-1}(gy_i)$ is $N_G(G_{g y_i})=G_{g y_i}$.

Define $X$ to be the $G$-CW-complex defined by $Z/\sim$, where $\sim$ is the relation generated by $x\sim y$ if and only if $r(x)=r(y)$. Hence $X$ is $G$-homotopically equivalent to $Z$. Therefore $X$ is a model for $\underline{E}
G$. Clearly $X$ is cocompact and by construction $X^H$ consists of exactly one point if $H$ is a non-trivial finite subgroup of $X$.


\end{proof}

Now we are ready to prove our main theorem.

\begin{theorem}\label{thm:M:NM:vcd:gd}
Let $G$ be a virtually torsion-free group of type $\barf$ that satisfies properties (M) and (NM).  Then $\vcd(G)=\cdfin(G).$
\end{theorem}
\begin{proof}
If $G$ is finite then there is nothing to prove. From now on we assume $G$ is infinite.

By \Cref{lemma:collapsing:fixpoints3} there exists a cocompact model $X$  for $\underline{E}G$ satisfying that $X^H$ consists of one point for every non-trivial finite subgroup $H$ of $G$. 


On the other hand, since $G$ is infinite and virtually torsion free, we conclude that $\vcd(G)>0$. Hence we have $\dim(X^F)=0< \vcd(G)$ for every non-trivial finite subgroup $F$ of $G$. Therefore by \Cref{second:criterion}, we have $\vcd(G)=\cdfin(G).$ 
\end{proof}

\section{Examples}

Next, we will describe some examples of groups satisfying the hypothesis of \Cref{thm:M:NM:vcd:gd}.

\subsection{Groups in the literature}
\begin{enumerate}
    \item Extensions $1\to \dbZ^n \to G\to F\to 1$ such that $F$ is finite and the conjugation action of $F$ on $\dbZ^n$ is free outside $0\in \dbZ^n$, and $G$ is of type $\barf$.  Properties (M) and (NM) for this groups are stablished in  \cite{DL03}. 
    
    It is worth noticing that, in the more general context of virtually poly-cyclic groups, it is already known that the Bredon cohomological dimension and the virtual cohomological dimension are equal, see for instance \cite[Example~5.26]{Lu05}.
    \item Cocompact Fuchsian groups and arithmetic Fuchsian groups. Let $G$ be a Fuchsian group, i.e.  $G$ acts properly discontinuously and by orientation-preserving isometries on the hyperbolic plane $\dbH$. A subgroup of $G$ is finite and non-trivial if and only if it fixes a unique point in $\dbH$. Also any element of infinite order does not fix any point of $\dbH$ because the action is proper. This implies that $\dbH$ is a model for $\underline{E}G$ such that the point set $\dbH^H$ consists of one point for every non-trivial finite subgroup $H$ of $G$.  Thus by \cref{lemma:M:NM2}, we have that $G$ satisfies properties (M) and (NM).  If additionally $G$ acts cocompactly on $\dbH$, then clearly satisfies the hypothesis of \cref{thm:M:NM:vcd:gd}. If $G$ is an arithmetic Fuchsian group, then the Borel-Serre bordification of $\dbH$ is a cocompact model for $\underline{E}G$ with $X^H$ is a one-point space for $H$ finite non-trivial. For more information about Fuchsian groups see \cite{freitag}.
    \item One relator groups admiting a cocompact model for $\underline{E}G$. Properties (M) and (NM) are verified in \cite{DL03}.
    
    A few comments on the existence of a model for $\underline{E}G$ for one relator groups. For 1-relator groups without torsion, the Cayley 2-complex is an $\underline{E}G = EG$. For 1-
relator groups with torsion, if the relator has the form $w^n$, then the torsion is all conjugate
into the cyclic subgroup generated by $w$, which has order $n$. If you take the version of
the Cayley 2-complex which has just one 2-cell with label $w^n$
 bounding each such word
in the Cayley graph (not the universal cover of the presentation 2-complex which would
have $n$ different 2-cells bounding each such word), this becomes a cocompact model for
$\underline{E}G$ provided that the 2-cells (whose stabilizer is cyclic of order $n$) are subdivided to ensure
that it becomes a $G$-CW-complex.
    \item The Hilbert modular group. A totally real number field $k$ is an algebraic extension of $\dbQ$ such that all its embeddings $\sigma_i:k\to \dbC$ have image contained in $\dbR$. Let $k$ denote a totally real number field of degree $n$ and
$\calO_k$ its ring of integers. The Hilbert modular group is by definition $PSL_2(\calO_k)$. If $K=\dbQ$ we recover the classical modular group $PSL_2(\dbZ)$. Properties (M) and (NM) are verified for the Hilbert modular group in \cite[Lemma~4.3]{BSS}. Since the Hilbert modular group is a lattice in $PSL_2(\dbR)\times\cdots \times PSL_2(\dbR)$, then it is an arithmetic group acting diagonally in the symmetric space $\dbH\times \cdots \times \dbH$. Hence the Borel-Serre bordification again provides a cocompact model for $\underline{E}G$. See \cite{freitag} for more information of the Hilbert modular group.
\end{enumerate}

\subsection{Groups acting on trees and properties (M) and (NM)}
 Let us quickly recall the notation of graph of groups from \cite{Se03}. A graph of groups $\mathbf{Y}$ consists of a graph $Y$ (in the sense of Serre), one group $Y_y$ for every edge $y$ of $Y$, one group $Y_P$ for each vertex $P$ of $Y$, and injective homomorphism $Y_e\to Y_P$ of $P$ is a vertex of the edge $y$. Recall that associated to $\mathbf{Y}$ we have the fundamental group $\pi_1(\mathbf{Y})$ and the Bass-Serre tree $T$, in such a way that $\pi_1(\mathbf{Y})$ acts on $T$ by simplicial automorphism and the quotient graph is isomorphic to $Y$. Denote by $f\colon T\to Y$ the quotient projection.
 
 We will be able to construct more examples using the following theorem.

\begin{theorem}\label{thm:graph:of:groups}
Let $\mathbf{Y}$ be a graph of groups in the sense of Serre with compact underlying graph. Assume that the vertex groups are of type $\barf$ and satisfy properties (M) and (NM), and assume that the edge groups are torsion free. Then the fundamental group $\pi_1(\mathbf{Y})$ of $Y$ is of type $\barf$ and satisfies properties (M) and (NM).
\end{theorem}
\begin{proof}
Let $T$ be the Bass-Serre tree of $\mathbf{Y}$. Denote $G=\pi_1(\mathbf{Y})$. Choose cocompact models $X_y$ and $X_P$ for $\underline{E}Y_y$ and $\underline{E}Y_P$ respectively. Then we can construct a model $X$ for $\underline{E}G$ as follows. Replace each vertex $v$ of $T$ by the corresponding $X_{f(v)}$, and each edge $e$ of $T$ by $X_{p(e)}\times [0,1]$. Next if $v$ is a vertex of $e$ glue one one of the $X_e\times {0}$ to $X_v$ using map induced by the homomorphism $X_{p(e)}\to X_{p(v)}$ (and $X_{p(e)}\times {1} \to X_{p(v')}$ where $v'$ is the other vertex of $e$). Hence $X$ inherits a $G$-action and we can easily verify that it is a model for $\underline{E}G$ (compare with \cite[Proposition~4.8]{JLSS}).  Moreover, the orbit space $X/G$ can be constructed using a similar construction using instead $Y$, $X_P/Y_P$ and $X_y/Y_y$. Therefore, since $Y$ is compact and each $X_P/Y_P$ and $X_y/Y_y$ are compact, we have that $X/G$ is also compact. This proves that $G$ admits a cocompact model for $\underline{E}G$.

Assume now that each $X_y$ and $X_P$ are models satisfying the conclusion of \cref{lemma:collapsing:fixpoints3}. Let $H$ be a non-trivial finite subgroup of $G$. Then $H$ cannot fix any edge of $T$, because every edge group of $\mathbf{Y}$ is torsion free. But, since $H$ is finite, has to fix one vertex of $T$. Hence $H$ fixes a unique vertex $v$ of $T$. Hence $H$ acts on $X_{p(v)}$, so $H$ fixes a unique point of $X_{p(v)}$. Therefore every non-trivial finite subgroup of $G$ fixes a unique point of $X$, and by \Cref{lemma:M:NM2} we conclude that $G$ satisfies properties (M) and (NM).
\end{proof}

Our final example are $3$-manifold groups. For more information about $3$-manifold groups, JSJ-decomposition, and the geometrization theorem see \cite{Mo05}.

\subsection{3-manifold groups} Let $M$ be a closed, orientable, connected $3$-manifold with fundamental group $G$. We claim that $G$ is of type $\barf$ and satisfies properties (M) and (NM). The prime decomposition $M=N_1\# \cdots \# N_m$ induces a splitting of $G=G_1*\cdots * G_m$. By \Cref{thm:graph:of:groups}, it is enough to prove that each $G_i$ is of type $\barf$ and satisfies properties (M) and (NM). 

From now on assume $M$ is prime. Using the Perelman-Thurston geometrization theorem we can chop off $M$ along tori to obtain pieces that are either hyperbolic or Sifert fibered. This is the so-called JSJ-decomposition. More explicitly, we can find a collection of tori (possibly empty) $T_1,\dots,T_r $ embedded in $M$ such that (abusing of notation) $M-\bigsqcup_i T_i$ is a disjoint union of manifolds (with boundary if the collection of tori is not empty) such that each piece is either hyperbolic or Seifert fibered.  Hence $G$ is the fundamental group of a graph of groups $\mathbf{Y}$ with vertex groups the fundamental groups of Seifert fibered manifolds or hyperbolic manifolds, and edge group isomorphic to $\dbZ^2$. Again, by \cref{thm:graph:of:groups}, it is enough to prove that all vertex groups in $\mathbf{Y}$ are of type $\barf$ and satisfy properties (M) and (NM). If the collection of tori is empty, then $M$ itself is either hyperbolic or Seifert fibered. If $M$ is hyperbolic, then $G$ is torsion free since $M$ is aspherical, thus $G$ satisfies properties (M) and (NM). Additionally  the universal cover of $M$ is a cocompact model for $\underline{E}G$. If $M$ is Seifert fibered, then $M$ is aspherical unless is covered by the three sphere $\dbS^3$ or by $\dbS^2\times \dbR$. In the $\dbS^3$ case $G$ is finite, while in the $\dbS^2\times \dbR$ case $G$ is either isomorphic to $\dbZ$ or to the infinite dihedral subgroup $D_\infty$. In both cases $G$ satisfies properties (M) and (NM) and is of type $\barf$. Finally, we have to deal with the case of a non-trivial JSJ-decomposition. In this case we can verify case by case that every hyperbolic and Seifert fibered manifold in the JSJ-decomposition is an aspherical manifolds, and therefore their fundamental groups are torsion free. Hence all vertex and edge groups of $\mathbf{Y}$ are torsion free, thus $G=\pi_1(\mathbf{Y})$ is torsion free. Also, we have that $M$ is a cocompact model for $G$. 

We can conclude that the fundamental group of every prime manifold is of type $\barf$ and satisfies properties (M) and (NM). Therefore every $3$-manifold group is of type $\barf$ and satisfies properties (M) and (NM).

\bibliographystyle{alpha} 
\bibliography{myblib}
\end{document}